\documentclass[10pt,a4paper]{article}
\usepackage{amsfonts}
\usepackage{latexsym}
\usepackage{cite}
\usepackage{amsmath,amsfonts,latexsym,amssymb}
\usepackage[mathscr]{eucal}
\usepackage{cases}
\usepackage{amsthm}

\usepackage[bf,small]{caption2}
\usepackage{float,color}
\usepackage{graphicx}
\usepackage{amsmath}
\usepackage{amssymb}
\usepackage[all]{xy}

\newtheorem{theorem}{theorem}[section]
\newtheorem{thm}[theorem]{Theorem}
\newtheorem{lem}[theorem]{Lemma}

\newtheorem{rmk}[theorem]{Remark}
\newtheorem{nota}[theorem]{Notation}

\begin{document}

\title{\textbf{Solving the isomorphism problems for two families of parafree groups}}
\author{\Large Haimiao Chen \\
{\normalsize Email: \emph{chenhm@math.pku.edu.cn}} \\
\normalsize \em{Beijing Technology and Business University, Beijing, China}}
\date{}
\maketitle

\begin{abstract}
  For any integers $m,n$ with $m\ne 0$ and $n>0$, let $G_{m,n}$ denote the group presented by $\langle x,y,z\mid x=[z^m,x][z^n,y]\rangle$; for any integers $m,n>0$, let $H_{m,n}$ denote the group presented by $\langle x,y,z\mid x=[x^m,z^n][y,z]\rangle$. By investigating cohomology jump loci of irreducible ${\rm GL}(2,\mathbb{C})$-character varieties, we show: if $m,m'\ne 0$, $n,n'>0$ and $G_{m',n'}\cong G_{m,n}$, then $m=m',n=n'$; if $m,m',n,n'>0$ and $H_{m',n'}\cong H_{m,n}$, then $m'=m, n'=n$.

\end{abstract}

\section{Introduction}

For nonzero integers $m,n$, let
\begin{align*}
G_{m,n}&=\langle x,y,z\mid x=[z^m,x][z^n,y]\rangle, \\
H_{m,n}&=\langle x,y,z\mid x=[x^m,z^n][y,z]\rangle,
\end{align*}
where the commutator $[x,y]=x^{-1}y^{-1}xy$.
The first family was introduced by Baumslag in 1960s \cite{Ba67,Ba69}, and the second family later \cite{Ba94}. They are examples of parafree groups. A group is called {\it parafree} if it is residually nilpotent and its lower central series of quotients are the same as those of some free group.

The isomorphism problem of groups is usually regarded to be fundamental. Taken seriously in Section 9 of \cite{Ba05} were three families, which we call the $G$, $H$, $K$ families. The isomorphism problem for the $K$ family was solved in \cite{HK17}, and those for the $G$ and $H$ families remain open. It was shown in \cite{FRS97} that $G_{m,1}\ncong G_{1,1}$ for $m>1$ and $G_{m,1}\ncong G_{m',1}$ for distinct prime $m,m'$. Enumerating homomorphisms to finite groups was applied in \cite{BCH04,LL94}, but could only deduce results for finitely many pairs $(m,n)$. Recently, infinitely members in a subfamily of the $G$ family have been distinguished from each other \cite{HK20}. Results on the $H$ family are rarely seen.

There is an isomorphism $G_{m,n}\cong G_{-m,-n}$ given by $x\mapsto x, y\mapsto y, z\mapsto z^{-1}$, so we may always assume $n>0$. (In some places, including \cite{Ba05}, $m,n$ are both assumed to be positive.)
There are isomorphisms $H_{m,n}\cong H_{-m-1,-n-1}$ and $H_{m,n}\cong H_{-m-1,n}$, given respectively by $x\mapsto x, y\mapsto yz^{n+1}x^{m+1}z^{-1}, z\mapsto z$, and $x\mapsto x^{-1}, y\mapsto yz^{1-n}, z\mapsto z^{-1}$.
Note that $H_{-1,n}\cong H_{m,-1}\cong F_2$, the free group on two generators. Thus for the $H$ family, we may assume $m,n>0$.

In this paper, we completely solve the isomorphism problems by showing
\begin{thm}\label{thm:solution-G}
For any integers $m,m',n,n'$ with $m,m'\ne 0$ and $n,n'>0$, if $G_{m,n}\cong G_{m',n'}$, then $m=m'$ and $n=n'$.
\end{thm}

\begin{thm}\label{thm:solution-H}
For any integers $m,m',n,n'>0$, if $H_{m,n}\cong H_{m',n'}$, then $m=m'$ and $n=n'$.
\end{thm}

This is achieved by studying ${\rm GL}(2,\mathbb{C})$-character varieties.

Given a finitely presented group $\Gamma$, let $\mathcal{R}^{\rm irr}(\Gamma)$ denote the space of irreducible representations $\Gamma\to{\rm GL}(2,\mathbb{C})$, on which ${\rm GL}(2,\mathbb{C})$ acts by conjugation. For each $\rho\in\mathcal{R}^{\rm irr}(\Gamma)$, its {\it character} $\chi_\rho$ is by definition the function $\Gamma\to\mathbb{C}$ sending $g$ to ${\rm tr}(\rho(g))$. Put $\mathcal{X}^{\rm irr}(\Gamma)=\{\chi_\rho\colon\rho\in\mathcal{R}^{\rm irr}(\Gamma)\}$, and call it the (irreducible) {\it ${\rm GL}(2,\mathbb{C})$-character variety} of $\Gamma$. It is known that $\rho,\rho'\in\mathcal{R}^{\rm irr}(\Gamma)$ are conjugated if and only if $\chi_\rho=\chi_{\rho'}$, so we may identify the conjugacy class $[\rho]$ with $\chi_\rho$. Also known is that $\mathcal{X}^{\rm irr}$ defines a functor from the category of groups to that of algebraic varieties. For more information, see \cite{LM85,Si12} and the references therein.

For each $\rho\in\mathcal{R}^{\rm irr}(\Gamma)$, let $V_\rho=\mathbb{C}^2$, equipped with the $\Gamma$-module structure via $\rho$. Then $\dim H^1(\Gamma;V_\rho)$ depends only on $[\rho]$.
For $k\in\mathbb{N}$, let
$$\mathcal{J}_k(\Gamma)=\{\chi_\rho\in\mathcal{X}^{\rm irr}(\Gamma)\colon \dim H^1(\Gamma;V_\rho)\ge k\}.$$
It is a subvariety of $\mathcal{X}^{\rm irr}(\Gamma)$, called the (nonabelian) {\it cohomology jump locus} or {\it characteristic variety} in degree 1 and depth $k$, as in the literature \cite{DP14,DPS09}.
The original (abelian) cohomology jump locus was defined as
$$\mathcal{V}_k^i(\Gamma)=\{\tau\in\hom(\Gamma,\mathbb{C}^\ast)\colon \dim H^i(\Gamma;V_\tau)\ge k\}$$
for nonzero integers $i,k$, and has found many applications (see \cite{DS09,PS10,SYZ15}, etc).

Here is the strategy for proving Theorem \ref{thm:solution-G} and Theorem \ref{thm:solution-H}. As a key observation, the morphism
$${\det}_\ast:\mathcal{X}^{\rm irr}(\Gamma)\to\Gamma^\wedge:=\hom(\Gamma,\mathbb{C}^\ast)$$
induced by $\det:{\rm GL}(2,\mathbb{C})\to\mathbb{C}^\ast$ is well-defined and is natural in $\Gamma$.
If there exists an isomorphism $\phi:\Gamma'\to\Gamma$, then the following diagram commutes:
\begin{align}
\xymatrix{
&\mathcal{J}_3(\Gamma)\ar[r]^{\phi^\ast}_{\cong}\ar[d]_{\det_\ast}&\mathcal{J}_3(\Gamma')\ar[d]^{\det_\ast} \\
&\Gamma^\wedge\ar[r]^{\phi^\ast}_{\cong}&(\Gamma')^\wedge
}  \label{eq:natural-transformation}
\end{align}
In particular, $\det_\ast(\phi^\ast(\mathcal{C}))\cong\det_\ast(\mathcal{C})$ for each component $\mathcal{C}\subset\mathcal{J}_3(\Gamma)$, and $\det_\ast^{-1}(\phi^\ast(\mathfrak{a}))\cong\det_\ast^{-1}(\mathfrak{a})$ for each $\mathfrak{a}\in\Gamma^\wedge$.
We are able to well understand the topology of each connected component of $\mathcal{J}_3(G_{m,n})$ and of each fiber of $\det_\ast$, and extract enough numerical information on $m,n$, to deduce $m'=m$ and $n'=n$ from the existence of an isomorphism $G_{m',n'}\cong G_{m,n}$. The same method is also successfully applied to the $H$ family.


In general, the whole character variety $\mathcal{X}^{\rm irr}(\Gamma)$ may be complicated. But $\mathcal{J}_k(\Gamma)$ has lower dimension, and is relatively easy to determine. The ${\rm GL}(2,\mathbb{C})$-character variety of a free group has a natural coordinate system via trace functions (see \cite{Dr03} for instance). For a 3-generator group $\Gamma$, we could have described $\mathcal{X}^{\rm irr}(\Gamma)$ as a subvariety of $\mathcal{X}^{\rm irr}(F_3)$ (which can be embedded as a hypersurface of $\mathbb{C}^9$ via trace coordinates). Nevertheless, we choose not to do so; instead, we choose a unique explicit representative for each conjugacy class of ${\rm GL}(2,\mathbb{C})$-representations, making computations more convenient. In this paper, ``varieties" are almost treated just as subspaces of $\mathbb{C}^N$ for some $N$.

\begin{nota}
\rm Let 
\begin{align*}
\mathbf{e}=\left(\begin{array}{cc} 1 & 0 \\ 0 & 1 \end{array}\right), \qquad
\mathbf{p}=\left(\begin{array}{cc} 1 & 1 \\ 0 & 1 \end{array}\right); \qquad 
\mathbf{d}(\mu,\nu)=\left(\begin{array}{cc} \mu & 0 \\ 0 & \nu \end{array}\right), \quad \mu,\nu\in\mathbb{C}^\ast.
\end{align*}

For $\mathbf{s}\in{\rm GL}(2,\mathbb{C})$, denote its $(i,j)$-entry by $\mathbf{s}_{ij}$, and its $j$-th column by $\mathbf{s}_{\ast j}$.

For $k\in\mathbb{Z}$ and an element $\mathfrak{r}$ of some ring with unit, let
$[k]_{\mathfrak{r}}=\sum_{j=0}^{k-1}\mathfrak{r}^j$ if $k>0$, and $[k]_{\mathfrak{r}}=-\mathfrak{r}^{k}\sum_{j=0}^{|k|-1}\mathfrak{r}^j$ if $k<0$; let $[0]_{\mathfrak{r}}=0$.

For a finite set $X$, let $\#X$ denote its cardinality.
\end{nota}

\section{The $G$ family}

In this section, let $g={\rm gcd}(m,n)$, the greatest common divisor of $m$ and $n$.
For a positive integer $k$, let $\Lambda_k=\{1\ne\zeta\in\mathbb{C}\colon \zeta^k=1\}$.

Identify $(G_{m,n})^\wedge$ with $\mathbb{C}^\ast\times\mathbb{C}^\ast$, via $\tau\mapsto(\tau(z),\tau(y))$.

The relation $x=[z^m,x][z^n,y]$ can be rewritten as
$$z^{n}w^{-1}z^{-m}w^2=y^{-1}z^{n}y, \qquad \text{with}\quad w=z^mx.$$
Given $\mathbf{z},\mathbf{w},\mathbf{y}\in{\rm GL}(2,\mathbb{C})$, there exists a representation $G_{m,n}\to{\rm GL}(2,\mathbb{C})$ sending $z,w,y$ respectively to $\mathbf{z},\mathbf{w},\mathbf{y}$ if and only if
\begin{align}
\mathbf{v}:=\mathbf{z}^{n}\mathbf{w}^{-1}\mathbf{z}^{-m}\mathbf{w}^2=\mathbf{y}^{-1}\mathbf{z}^{n}\mathbf{y}; \label{eq:G-relation}
\end{align}
when this holds, denote the unique representation by $\rho_{\mathbf{z},\mathbf{w},\mathbf{y}}$. It is irreducible if and only if $\mathbf{z},\mathbf{w},\mathbf{y}$ have no common eigenvector.

Suppose $\rho=\rho_{\mathbf{z},\mathbf{w},\mathbf{y}}$ is an irreducible representation.

If $\mathfrak{d}:G_{m,n}\to V_\rho$ is a derivation,
with $\mathfrak{d}(w)=\xi_1, \mathfrak{d}(y)=\xi_2, \mathfrak{d}(z)=\xi_3$,
then
\begin{align*}
\mathfrak{d}(z^nw^{-1}z^{-m}w^2)
&=\mathbf{z}^n\mathbf{w}^{-1}\mathbf{z}^{-m}(\mathbf{e}+\mathbf{w}-\mathbf{z}^m)\xi_1
+([n]_{\mathbf{z}}+\mathbf{z}^n\mathbf{w}^{-1}[-m]_{\mathbf{z}})\xi_3, \\
\mathfrak{d}(y^{-1}z^ny)&=\mathbf{y}^{-1}(\mathbf{z}^n-\mathbf{e})\xi_2+\mathbf{y}^{-1}[n]_{\mathbf{z}}\xi_3.
\end{align*}
Hence the space $\mathfrak{D}_\rho$ of derivations $G_{m,n}\to V_\rho$ can be identified with that of triples $(\xi_1,\xi_2,\xi_3)$ satisfying
\begin{align}
&\mathbf{z}^n\mathbf{w}^{-1}\mathbf{z}^{-m}(\mathbf{e}+\mathbf{w}-\mathbf{z}^m)\xi_1+\mathbf{y}^{-1}(\mathbf{e}-\mathbf{z}^n)\xi_2 \nonumber \\
&\qquad\quad +(\mathbf{z}^n\mathbf{w}^{-1}[-m]_{\mathbf{z}}+[n]_{\mathbf{z}}-\mathbf{y}^{-1}[n]_{\mathbf{z}})\xi_3=0.     \label{eq:relation0}
\end{align}
Under this identification, the subspace of inner derivations
$$\mathfrak{I}_\rho=\{((\mathbf{w}-\mathbf{e})\eta,(\mathbf{y}-\mathbf{e})\eta,(\mathbf{z}-\mathbf{e})\eta)\colon\eta\in V_\rho\},$$
which, due to the irreducibility of $\rho$, is 2-dimensional.
According to the basic fact $H^1(G_{m,n};V_{\rho})\cong\mathfrak{D}_\rho/\mathfrak{I}_\rho$ (see \cite{HS97} Chap. VI, Corollary 5.2), $\chi_\rho\in\mathcal{J}_3(G_{m,n})$ if and only if $\dim\mathfrak{D}_\rho\ge 5$.

Multiplied by $\mathbf{z}^m\mathbf{w}\mathbf{z}^{-n}$ on the left, (\ref{eq:relation0}) becomes
$\mathbf{a}\xi_1+\mathbf{c}\xi_3=\mathbf{b}(\mathbf{y}^{-1}\xi_2)$, with
\begin{align*}
\mathbf{a}&=\mathbf{e}+\mathbf{w}-\mathbf{z}^m, \\
\mathbf{b}&=\mathbf{w}^2-\mathbf{z}^m\mathbf{w}\mathbf{z}^{-n}, \\
\mathbf{c}&=\mathbf{z}^m[-m]_{\mathbf{z}}+\mathbf{z}^m\mathbf{w}\mathbf{z}^{-n}[n]_{\mathbf{z}}
-\mathbf{w}^2\mathbf{y}^{-1}\mathbf{z}^{-n}[n]_{\mathbf{z}}.
\end{align*}
Hence $\dim \mathfrak{D}_\rho\ge 5$ if and only if ${\rm rank}(\mathbf{a},\mathbf{b},\mathbf{c})\le 1$.

Regarding (\ref{eq:relation0}), a necessary condition for $\dim \mathfrak{D}_\rho\ge 5$ is $\det(\mathbf{z}^n-\mathbf{e})=0$. Also necessary is $\mathbf{z}^n\ne\mathbf{e}$: otherwise, by (\ref{eq:G-relation}), $\mathbf{w}=\mathbf{z}^m$ so that $\mathbf{a}=\mathbf{e}$, which would imply $\dim \mathfrak{D}_\rho=4$.
Therefore, up to conjugacy we may assume that either $\mathbf{z}=\mathbf{d}(\lambda,\zeta)$ with $\zeta^n=1\ne\lambda^n$, or $\mathbf{z}=\zeta\mathbf{p}$ with $\zeta^n=1$.
Suppose
$$\mathbf{w}=\left(\begin{array}{cc} a & b \\ c & d \end{array}\right), \qquad
\mathbf{y}=\left(\begin{array}{cc} a' & b' \\ c' & d' \end{array}\right),$$
with $ad-bc=\lambda^m\zeta^m$ and $t:=a'd'-b'c'\ne 0$.

\begin{rmk} \label{rmk:G-condition}
\rm Now that $\mathfrak{I}_\rho\subset\mathfrak{D}_\rho$, we have
$$\mathbf{c}(\mathbf{z}-\mathbf{e})\eta=\mathbf{a}(\mathbf{e}-\mathbf{w})\eta+\mathbf{b}(\mathbf{e}-\mathbf{y}^{-1})\eta$$
for all $\eta\in V_\rho$.
Consequently, when $\zeta\ne 1$ so that $\mathbf{z}-\mathbf{e}$ is invertible, ${\rm rank}(\mathbf{a},\mathbf{b})\le 1$ has been sufficient for $\dim \mathfrak{D}_\rho\ge 5$; when $\zeta=1$, the condition ${\rm rank}(\mathbf{a},\mathbf{b},\mathbf{c}_{\ast2})\le 1$ is sufficient for $\dim \mathfrak{D}_\rho\ge 5$.
\end{rmk}

\subsection{$\mathbf{z}=\mathbf{d}(\lambda,\zeta)$ with $\zeta^n=1\ne\lambda^{n}$}

Let $r=a+d$, and introduce $\vartheta_k=\lambda^k-1$ for $k\in\mathbb{Z}$.

There exists $\mathbf{y}$ with $\mathbf{v}=\mathbf{y}^{-1}\mathbf{z}^n\mathbf{y}$ if and only if ${\rm tr}(\mathbf{v})={\rm tr}(\mathbf{z}^n)$, which is equivalent to
\begin{align}
\vartheta_n(\lambda^{-m}-\zeta^{-m})rad+(\lambda^{m}\vartheta_n+\zeta^m)a+(\lambda^{n+m}-\vartheta_n\zeta^m)d=(\lambda^n+1)\lambda^m\zeta^m. \label{eq:trace}
\end{align}
We can write $\mathbf{y}\mathbf{v}=\mathbf{z}^n\mathbf{y}$ as
\begin{align}
(a',b')(\mathbf{v}-\lambda^n\mathbf{e})=(c',d')(\mathbf{v}-\mathbf{e})=0.  \label{eq:determine-y}
\end{align}
Writing $(\mathbf{a},\mathbf{b})$ explicitly as
$$\left(\begin{array}{cccc}
a-\vartheta_m & b & (r-\lambda^{m-n})a-\lambda^m\zeta^m & (r-\lambda^{m})b \\
c & d+1-\zeta^m & (r-\zeta^m\lambda^{-n})c & (r-\zeta^{m})d-\lambda^m\zeta^m
\end{array}\right),$$
we easily see that ${\rm rank}(\mathbf{a},\mathbf{b})\le 1$ if and only if
\begin{align*}
{\rm rank}\left(\begin{array}{cccc} a-\vartheta_m & b & \kappa_1 & 0 \\
c & d+1-\zeta^m & 0 & \kappa_2 \end{array}\right)\le 1,
\end{align*}
where
\begin{align*}
\kappa_1&=(\zeta^m\lambda^{-n}-\lambda^{m-n}+\vartheta_m)a+\vartheta_md+\zeta^m(\lambda^{-n}-\lambda^{m-n}-\lambda^{m}), \\
\kappa_2&=(\zeta^m-1)a+\vartheta_md+\lambda^m(1-2\zeta^m).
\end{align*}

If $\kappa_1\ne 0$, then $c=0$, $d=\zeta^m-1$, and $\kappa_2=0$, i.e.
$$0=\lambda^m\zeta^m+\vartheta_m(\zeta^m-1)+\lambda^m(1-2\zeta^m)=1-\zeta^m=-d,$$
which is absurd.
Thus $\kappa_1=0$. 

\subsubsection{$\kappa_2=0$}

The condition $\det(\mathbf{a})=0$ reads
$$(1-\zeta^{m})a-\vartheta_md+(2\lambda^m-1)\zeta^m-\vartheta_m=0,$$
which together with $\kappa_2=0$ implies $\zeta^m=1$ (so that $\zeta^g=1$), and
$d=\lambda^m\vartheta_m^{-1}$. Then $\kappa_1=0$ becomes $a=\vartheta_n^{-1}$. As a result,
$$r=\vartheta_{m+n}\vartheta_m^{-1}\vartheta_n^{-1}, \qquad bc=ad-\lambda^m=\lambda^m(\vartheta_m^{-1}\vartheta_n^{-1}-1),$$
and
\begin{align*}
\mathbf{v}&=\lambda^{-m}\left(\begin{array}{cc}  \lambda^n(\lambda^{-m}ra-1)d-\lambda^nrbc & (\lambda^m+\vartheta_mrd)\lambda^nb \\ (1-\vartheta_{-m}ra)c & a(rd-\lambda^m)-\lambda^{-m}rbc \end{array}\right) \\
&=\left(\begin{array}{cc}
(\lambda^{n}(\lambda^n\vartheta_m\vartheta_n+\vartheta_{-m}-\vartheta_n)\vartheta_m^{-1}\vartheta_n^{-2} & \lambda^n(1-\lambda^{-m}\vartheta_{m+n}\vartheta_m^{-1}\vartheta_n^{-1})b \\
\lambda^{-m}(\lambda^n\vartheta_n-\vartheta_{-m})\vartheta_n^{-2}c & (\lambda^{n-m}\vartheta_{m+n}\vartheta_m^{-1}\vartheta_n^{-1}-1)\vartheta_n^{-1} \end{array}\right).
\end{align*}

It turns out that (\ref{eq:trace}) has been fulfilled: 
\begin{align*}
&\vartheta_n\vartheta_{-m}\cdot\vartheta_{m+n}\vartheta_m^{-1}\vartheta_n^{-1}\cdot\vartheta_n^{-1}\cdot\lambda^m\vartheta_m^{-1}+(\lambda^m\vartheta_n+1)\vartheta_n^{-1}+(\lambda^{n+m}-\vartheta_n)\lambda^m\vartheta_m^{-1} \\
=\ &-\vartheta_{m+n}\vartheta_m^{-1}\vartheta_n^{-1}+\lambda^m+\vartheta_n^{-1}+(\lambda^{n+m}-\vartheta_n)\lambda^m\vartheta_m^{-1}=(\lambda^n+1)\lambda^m.
\end{align*}
Observe that $(a,b)(\mathbf{v}-\mathbf{e})=0$: clearly $(a,b)(\mathbf{v}-\mathbf{e})_{\ast 2}=0$; not obvious is
\begin{align*}
&(a,b)(\mathbf{v}-\mathbf{e})_{\ast 1} \\
=\ &\lambda^n(\lambda^n\vartheta_m\vartheta_n+\vartheta_{-m}-\vartheta_n)\vartheta_m^{-1}\vartheta_n^{-3}-\vartheta_n^{-1}+(\lambda^n\vartheta_n-\vartheta_{-m})\vartheta_n^{-2}(\vartheta_m^{-1}\vartheta_n^{-1}-1) \\
=\ &\big((\lambda^n\vartheta_n+1)\vartheta_m\vartheta_n+\lambda^n(\vartheta_{-m}-\vartheta_n)+(\lambda^n\vartheta_n-\vartheta_{-m})(1-\vartheta_m\vartheta_n)\big)\vartheta_m^{-1}\vartheta_n^{-3}=0.
\end{align*}
Hence by (\ref{eq:determine-y}) we have $(c',d')\parallel(a,b)$.
Fix the conjugacy indeterminacy of $\rho$ by setting $c'=1$. Then $d'=\vartheta_nb$.
Note that the case $b'=b=0$ never occur, as $\vartheta_na'b-b'=t\ne 0$. So the irreducibility of $\rho_{\mathbf{z},\mathbf{w},\mathbf{y}}$ is already ensured.

From (\ref{eq:determine-y}) we see
\begin{align*}
0&=(a',a'd'-t)(\mathbf{v}-\lambda^n\mathbf{e})=a'(1,d')(\mathbf{v}-\lambda^n\mathbf{e})-(0,t)(\mathbf{v}-\lambda^n\mathbf{e}) \\
&=-\vartheta_na'(1,d')-(0,t)(\mathbf{v}-\lambda^n\mathbf{e}),
\end{align*}
and then deduce
$$a'=\lambda^{-m}(\vartheta_{-m}-\lambda^n\vartheta_n)\vartheta_n^{-3}ct,  \qquad  b'=(\lambda^n\vartheta_n-\vartheta_{-m})(1-\vartheta_m^{-1}\vartheta_n^{-1})\vartheta_n^{-2}t.$$
Let $\varrho(\zeta;\lambda,b,c,t)=\rho_{\mathbf{z},\mathbf{w},\mathbf{y}}$, with
\begin{align*}
\mathbf{z}=\mathbf{d}(\lambda,\zeta), \qquad \mathbf{w}=\left(\begin{array}{cc} \vartheta_n^{-1} & b \\ c & \lambda^{m}\vartheta_m^{-1} \end{array}\right),  \\
\mathbf{y}=\left(\begin{array}{cc} \lambda^{-m}(\vartheta_{-m}-\lambda^n\vartheta_n)\vartheta_n^{-3}ct & ((\lambda^n\vartheta_n-\vartheta_{-m})(1-\vartheta_m^{-1}\vartheta_n^{-1})\vartheta_n^{-2}-1)t \\
1 & \vartheta_nb \end{array}\right).
\end{align*}

When $\zeta\ne 1$, there is no further constraint.
Let
$$\mathcal{F}_\zeta=\{\varrho(\zeta;\lambda,b,c,t)\colon\lambda, t\in\mathbb{C}^\ast,\ \vartheta_m\vartheta_n\ne 0,\ bc=\lambda^m(\vartheta_m^{-1}\vartheta_n^{-1}-1)\}\subset\mathcal{R}^{\rm irr}(G_{m,n}).$$


\begin{lem}\label{lem:f}
The function $f(\lambda):=\vartheta_m\vartheta_n-1=\lambda^{n+m}-\lambda^{n}-\lambda^m$ has no multiple root, so the number of distinct roots of $f$ is $\ell_{m,n}:=\max\{m,n,n-m\}$.
\end{lem}
\begin{proof}
Assume $f(\lambda)=f'(\lambda)=0$, i.e.
$$\lambda^{n+m}=\lambda^{n}+\lambda^m, \qquad  (n+m)\lambda^{n+m}=n\lambda^n+m\lambda^m.$$
Then $\lambda^m=1-m/n$, and $\lambda^n=1-n/m$. Clearly, $m\ne n$.
\begin{itemize}
  \item If $m>n$, then $|\lambda|^n<1$, implying $|\lambda|<1$, so $m/n-1=|\lambda|^m<|\lambda|^n=1-n/m$, which is absurd.
  \item If $0<m<n$, then $|\lambda|^m<1$, implying $|\lambda|<1$, so $1-m/n=|\lambda|^m<|\lambda|^n=n/m-1$, which is also absurd.
  \item If $m<0$, then $|\lambda|^m=1-m/n>1$, implying $|\lambda|<1$, but on the other hand, $|\lambda|^n=1-n/m>1$. This is a contradiction.
\end{itemize}
\end{proof}

When $\zeta=1$, direct computation gives
\begin{align*}
\mathbf{c}_{\ast2}=\left(\begin{array}{cc}
n\lambda^{m}b+nt^{-1}(r(b'a-a'b)-\lambda^mb') \\ nt^{-1}(r(b'c-a'd)+\lambda^ma')+nd-m \end{array}\right).
\end{align*}
It can be verified that
$$\mathbf{c}_{\ast2}\parallel\mathbf{a}_{\ast2}=(b,d)^{\rm tr}\Leftrightarrow a'b-b'a=(1+m\lambda^{-m}/n)bt,$$
in which case
\begin{align*}
b=\frac{(a'd'-b')a}{(1+m\lambda^{-m}/n)t}=\frac{n\lambda^{m}}{(n\lambda^{m}+m)\vartheta_n}=:b_\lambda,
\end{align*}
and $c=c_\lambda:=(\vartheta_m^{-1}-\vartheta_n)(\lambda^m+m/n)$. Note that $\lambda^{m}\ne -m/n$ is required.

Let
$$\mathcal{F}_1=\big\{\varrho(1;\lambda,b_\lambda,c_\lambda,t)\colon\lambda,t\in\mathbb{C}^\ast,\ \vartheta_m\vartheta_n\ne 0,\ \lambda^m\ne -m/n\big\}\subset\mathcal{R}^{\rm irr}(G_{m,n}).$$


\subsubsection{$\kappa_2\ne 0$}

In this case, we have $a-\vartheta_m=b=0$, so that $a=\vartheta_m$, $d=\lambda^m\zeta^m\vartheta_m^{-1}$, and then $\kappa_1=0$ implies
$f(\lambda)=0$.
Thus, alternatively, $a=\lambda^{m-n}, d=\lambda^n\zeta^m$.

It turns out that (\ref{eq:trace}) has been fulfilled.

By direct computation,
$$\mathbf{v}=\left(\begin{array}{cc} 1 & 0 \\ \zeta^{-m}(\lambda^{m-2n}\zeta^{-m}-\vartheta_{-2n})c & \lambda^{n} \end{array}\right).$$
It follows from (\ref{eq:determine-y}) that $d'=0$, and
\begin{align*}
(a',b')\parallel\big((\lambda^{m-2n}\zeta^{-m}-\vartheta_{-2n})c,\zeta^{m}\vartheta_n\big).
\end{align*}
Fix the conjugacy indeterminacy by setting $c'=1$. Then
$$b'=-t, \qquad  a'=(\vartheta_{-2n}-\lambda^{m-2n}\zeta^{-m})\zeta^{-m}\vartheta_n^{-1}tc.$$
If $\zeta=1$, then by computation, $\mathbf{c}_{12}=-nt\lambda^{2m-2n}\ne 0$, violating ${\rm rank}(\mathbf{a},\mathbf{b},\mathbf{c}_{\ast2})\le 1$.
Hence $\zeta=1$ is forbidden.

Note that
$$0\ne\kappa_2=(\zeta^m-1)\lambda^{m-n}+\lambda^m\zeta^m+\lambda^m(1-2\zeta^m)=1-\zeta^m.$$
For each $\zeta\in\Lambda_n-\Lambda_g$ and each $\lambda\in f^{-1}(0)$, let
$$\mathcal{G}_{\zeta,\lambda}=\{\varrho'(\zeta;\lambda,c,t)\colon c\in\mathbb{C},\ t\in\mathbb{C}^\ast\}\subset\mathcal{R}^{\rm irr}(G_{m,n}),$$
where $\varrho'(\zeta;\lambda,c,t)=\rho_{\mathbf{z},\mathbf{w},\mathbf{y}}$, given by $\mathbf{z}=\mathbf{d}(\lambda,\zeta)$,
$$\mathbf{w}=\left(\begin{array}{cc}  \lambda^{m-n} & 0 \\ c & \lambda^{n}\zeta^m \end{array}\right), \qquad
\mathbf{y}=\left(\begin{array}{cc}  (\vartheta_{-2n}-\lambda^{m-2n}\zeta^{-m})\zeta^{-m}\vartheta_n^{-1}tc & -t \\ 1 & 0 \end{array}\right).$$

\subsection{$\mathbf{z}=\zeta\mathbf{p}$ with $\zeta^n=1$} \label{sec:G-p}


If $c=0$, then $\mathbf{v}_{21}=(\mathbf{y}^{-1}\mathbf{z}^n\mathbf{y})_{21}$ would imply $c'=0$, too, contradicting the irreducibility of $\rho$. Hence $c\ne 0$.

Since the upper-left entry of
$$\left(\begin{array}{cc} 1 & -ac^{-1} \\ 0 & 1 \end{array}\right)\left(\begin{array}{cc} a & b \\ c & d \end{array}\right)\left(\begin{array}{cc} 1 & -ac^{-1} \\ 0 & 1 \end{array}\right)^{-1}$$
vanishes, we may just assume $a=0$ at the beginning. This also fixes the conjugacy indeterminacy.

Write $(\mathbf{a},\mathbf{b})$ explicitly as
$$\left(\begin{array}{cccc}
1-\zeta^m & b-m\zeta^m & -\zeta^{m}(\zeta^m+mc) & (d-\zeta^m)b+\zeta^mm(nc-d) \\
c & d+1-\zeta^m & (d-\zeta^m)c & d^2-\zeta^{2m}+\zeta^m(nc-d) \end{array}\right).$$
The conditions $\mathbf{a}_{\ast1}\parallel\mathbf{a}_{\ast2}$ and $\mathbf{a}_{\ast1}\parallel\mathbf{b}_{\ast1}$ respectively read
\begin{align*}
(\zeta^m-1)(d+1-\zeta^m)&=\zeta^{2m}+m\zeta^mc, \\
(\zeta^m-1)(d-\zeta^m)&=\zeta^{2m}+m\zeta^mc,
\end{align*}
which imply $\zeta^m=1$, and $c=-1/m$, so that $b=m$. Furthermore, $\mathbf{b}_{\ast2}\parallel\mathbf{a}_{\ast1}$ forces $m=-n$.

Now $\mathbf{v}=\mathbf{y}^{-1}\mathbf{z}^{n}\mathbf{y}$ becomes
$$\left(\begin{array}{cc} 1-d+d^2 & nd^3-nd^2 \\ (1-d)/n & 1+d-d^2 \end{array}\right)
=\left(\begin{array}{cc} 1+nc'd't^{-1} & n(d')^2t^{-1} \\ -n(c')^2t^{-1} & 1-nc'd't^{-1} \end{array}\right),$$
which is equivalent to
$t(d-1)=n^2(c')^2$ and $d'=nc'd$, forcing $c'\ne 0$, $d\ne 1$.

Let $\sigma(\zeta;u,v,t)$ denote $\rho_{\mathbf{z},\mathbf{w},\mathbf{y}}$, given by $\mathbf{z}=\zeta\mathbf{p}$,
\begin{align*}
\mathbf{w}=\left(\begin{array}{cc} 0 & -n \\ 1/n & n^2v^2t^{-1}+1 \end{array}\right), \qquad
\mathbf{y}=\left(\begin{array}{cc} u & n^3uv^2t^{-1}+nu-tv^{-1} \\ v & n^3v^3t^{-1}+nv \end{array}\right).
\end{align*}

When $\zeta\ne 1$, let
$$\mathcal{H}_\zeta=\big\{\sigma(\zeta;u,v,t)\colon u\in\mathbb{C},\ v,t\in\mathbb{C}^\ast\big\}\subset\mathcal{R}^{\rm irr}(G_{m,n}).$$

When $\zeta=1$,
\begin{align*}
\mathbf{c}=(\mathbf{e}+\mathbf{z}^{-n}\mathbf{w}-\mathbf{w}^2\mathbf{y}^{-1})\mathbf{z}^{-n}[n]_{\mathbf{z}}
=\left(\begin{array}{cc} 0 & (c')^{-1}-n(d+1) \\ \star & \star \end{array}\right)\left(\begin{array}{cc} n & \star \\ 0 & n \end{array}\right),
\end{align*}
where each $\star$ stands for something irrelevant. Hence $\mathbf{c}_{12}=0$ is equivalent to $nc'(d+1)=1$.
Let
$$\mathcal{H}_1=\big\{\sigma(1;u,v,t)\colon u\in\mathbb{C},\ v\in\mathbb{C}^\ast-\{1/2n\},\ t=n^3v^3/(1-2nv)\big\}.$$

\subsection{The result}

Recall $f(\lambda)=\lambda^{n+m}-\lambda^n-\lambda^m$ and $\ell_{m,n}=\max\{m,n,n-m\}$ in Lemma \ref{lem:f}.

Recall the identification of $(G_{m,n})^\wedge$ with $\mathbb{C}^\ast\times\mathbb{C}^\ast$ via $\tau\mapsto(\tau(z),\tau(y))$.

When $m\ne -n$, decomposed into connected components,
$$\mathcal{J}_3(G_{m,n})=\bigsqcup_{\zeta:\zeta^g=1}\chi(\mathcal{F}_\zeta)
\sqcup\bigsqcup_{\zeta\in\Lambda_n-\Lambda_g}\bigsqcup_{\lambda\in f^{-1}(0)}\chi(\mathcal{G}_{\zeta,\lambda}).$$

More intricate is the description of the morphism $\det_\ast$.
\begin{itemize}
  \item For $\zeta\in\Lambda_g$,
        $$\chi(\mathcal{F}_\zeta)\cong\{(\lambda,b,c)\colon \lambda\in\mathbb{C}^\ast,\ \vartheta_m\vartheta_n\ne 0, \ bc=\lambda^m(\vartheta_m^{-1}\vartheta_n^{-1}-1)\}\times\mathbb{C}^\ast,$$
        through which
        ${\det}_\ast:\chi(\mathcal{F}_\zeta)\to(G_{m,n})^\wedge$ sends $(\lambda,b,c,t)$ to $(\lambda\zeta,t).$
        Thus
        $${\det}_\ast(\chi(\mathcal{F}_\zeta))\cong\{\lambda\in\mathbb{C}^\ast\colon\vartheta_m\vartheta_n\ne 0\}\times\mathbb{C}^\ast,$$ 
        and for each $\mathfrak{a}=(\lambda\zeta,t)\in\det_\ast(\chi(\mathcal{F}_{\zeta}))$,
        $${\det}_\ast^{-1}(\mathfrak{a})\cong
        \begin{cases} \mathbb{C}^\ast, &\vartheta_m\vartheta_n\ne 1, \\
        \mathbb{C}\times\{0\}\cup\{0\}\times\mathbb{C}, &\vartheta_m\vartheta_n=1.
        \end{cases}$$
  \item $\chi(\mathcal{F}_1)\cong\{\lambda\colon \lambda\in\mathbb{C}^\ast,\ \vartheta_m\vartheta_n\ne 0,\ \lambda^m\ne -m/n\}\times\mathbb{C}^\ast,$
        through which ${\det}_\ast:\chi(\mathcal{F}_1)\to(G_{m,n})^\wedge$ sends $(\lambda,t)$ to $(\lambda,t).$
  \item For $\zeta\in\Lambda_n-\Lambda_g$ and $\lambda\in f^{-1}(0)$, we have
        $${\det}_\ast:\mathbb{C}\times\mathbb{C}^\ast\cong\chi(\mathcal{G}_{\zeta,\lambda})\to(G_{m,n})^\wedge, \qquad (c,t)\mapsto (\lambda\zeta,t).$$
        So $\det_\ast^{-1}(\mathfrak{a})\cong\mathbb{C}$ for each $\mathfrak{a}\in\det_\ast(\chi(\mathcal{G}_{\zeta,\lambda}))$.
\end{itemize}

When $m=-n$, 
$$\mathcal{J}_3(G_{-n,n})=\bigsqcup_{\zeta:\zeta^n=1}\chi(\mathcal{F}_\zeta)\sqcup\bigsqcup_{\zeta:\zeta^n=1}\chi(\mathcal{H}_\zeta).$$
\begin{itemize}
  \item For $\zeta\in\Lambda_n$, we have
        $${\det}_\ast:\mathbb{C}\times\mathbb{C}^\ast\times\mathbb{C}^\ast\cong\chi(\mathcal{H}_\zeta)\to(G_{m,n})^\wedge, \qquad (u,v,t)\mapsto(\zeta^2,t).$$
        Consequently, $\det_\ast^{-1}(\mathfrak{a})\cong\mathbb{C}\times\mathbb{C}^\ast$ for each $\mathfrak{a}\in\det_\ast(\chi(\mathcal{H}_{\zeta}))$.
  \item $\chi(\mathcal{H}_1)\cong\mathbb{C}\times(\mathbb{C}^\ast-\{1/2n\})$, through which
        ${\det}_\ast:\chi(\mathcal{H}_1)\to(G_{m,n})^\wedge$ sends $(u,v)$ to $(1,n^3v^3/(1-2nv))$. We have
        $${\det}_\ast^{-1}(\mathfrak{a})\cong\begin{cases}
        \mathbb{C}\sqcup\mathbb{C}\sqcup\mathbb{C}, &\mathfrak{a}\ne(1,-27/32), \\  \mathbb{C}\sqcup\mathbb{C}, &\mathfrak{a}=(1,-27/32).\end{cases}$$
\end{itemize}

\begin{proof}[Proof of Theorem \ref{thm:solution-G}]
Suppose there exists an isomorphism $\phi:G_{m',n'}\stackrel{\cong}\to G_{m,n}$.
We apply the naturality of $\det_\ast$ as displayed in (\ref{eq:natural-transformation}).

If there exists $\mathfrak{a}\in(G_{m',n'})^\wedge$ such that $\det_\ast^{-1}(\mathfrak{a})\cong\mathbb{C}\sqcup\mathbb{C}\sqcup\mathbb{C}$, then $m'=-n'$, and $\det_\ast^{-1}(\phi^\ast(\mathfrak{a}))\cong\mathbb{C}\sqcup\mathbb{C}\sqcup\mathbb{C}$, which implies $m=-n$.
Furthermore,
\begin{align*}
n-1&=\#\big\{\mathcal{C}\in\pi_0(\mathcal{J}_3(G_{m,n}))\colon\mathcal{C}\cong\mathbb{C}\times\mathbb{C}^\ast\times\mathbb{C}^\ast\big\}  \\
&=\#\big\{\mathcal{C}'\in\pi_0(\mathcal{J}_3(G_{m',n'}))\colon\mathcal{C}'\cong\mathbb{C}\times\mathbb{C}^\ast\times\mathbb{C}^\ast\big\}=n'-1.
\end{align*}
Hence $n=n'$.

Now suppose $m\ne -n$ and $m'\ne -n'$. Let $g'={\rm gcd}(m',n')$. 
\begin{itemize}
  \item The number of 3-dimensional components of $\mathcal{J}_3(G_{m,n})$ coincides with that of $\mathcal{J}_3(G_{m',n'})$, so $g=g'$.
  \item For each 3-dimensional component $\mathcal{C}\subset\mathcal{J}_3(G_{m,n})$,
        $$\mathbb{Z}^{n+|m|-g+2}\cong H_1({\det}_\ast(\mathcal{C}))\cong H_1({\det}_\ast(\phi^\ast\mathcal{C}))\cong\mathbb{Z}^{n'+|m'|-g'+2},$$
        implying $n+|m|=n'+|m'|$.
        Moreover,
        \begin{align*}
        \ell_{m,n}&=\#\pi_0(\{\mathfrak{a}\in{\det}_\ast(\mathcal{C})\colon{\det}_\ast^{-1}(\mathfrak{a})\cong\mathbb{C}\times\{0\}\cup\{0\}\times\mathbb{C}\}) \\
        &=\#\pi_0(\{\mathfrak{b}\in{\det}_\ast(\phi^\ast\mathcal{C})\colon{\det}_\ast^{-1}(\mathfrak{b})\cong\mathbb{C}\times\{0\}\cup\{0\}\times\mathbb{C}\})=\ell_{m',n'}.
        \end{align*}
  \item Comparing the numbers of components which are isomorphic to $\mathbb{C}\times\mathbb{C}^\ast$, we obtain $(n-g)\ell_{m,n}=(n'-g')\ell_{m',n'}$.
\end{itemize}
From all of these, we deduce $n=n'$ and $m=m'$.
\end{proof}

\section{The $H$ family}

Use $u=yz$ to present $H$ alternatively as
\begin{align*}
H_{m,n}=
\langle x,u,z\mid z^{-n-1}x^{-m}z^nx^{m+1}=u^{-1}z^{-1}u\rangle.
\end{align*}
Given $\mathbf{z},\mathbf{x},\mathbf{u}\in{\rm GL}(2,\mathbb{C})$, there exists a representation $H_{m,n}\to{\rm GL}(2,\mathbb{C})$ sending $z,x,u$ respectively to $\mathbf{z},\mathbf{x},\mathbf{u}$ if and only if
\begin{align}
\mathbf{v}:=\mathbf{z}^{-n-1}\mathbf{x}^{-m}\mathbf{z}^n\mathbf{x}^{m+1}=\mathbf{u}^{-1}\mathbf{z}^{-1}\mathbf{u}; \label{eq:H-relation}
\end{align}
when this holds, denote the unique representation by $\rho_{\mathbf{z},\mathbf{x},\mathbf{u}}$. It is irreducible if and only if $\mathbf{z},\mathbf{x},\mathbf{u}$ do not share an eigenvector.

Suppose $\rho=\rho_{\mathbf{z},\mathbf{x},\mathbf{u}}$ is irreducible. Clearly, $\mathbf{z}\ne\mathbf{e}$.

If $\mathfrak{d}:H_{m,n}\to V_\rho$ is a derivation, with $\mathfrak{d}(x)=\xi_1, \mathfrak{d}(u)=\xi_2, \mathfrak{d}(z)=\xi_3$,
then
\begin{align*}
\mathfrak{d}(z^{-n-1}x^{-m}z^nx^{m+1})&=\mathbf{z}^{-n-1}\mathbf{x}^{-m}\mathbf{g}\xi_1+\mathbf{h}\xi_3,  \\
\mathfrak{d}(u^{-1}z^{-1}u)&=\mathbf{u}^{-1}(\mathbf{z}^{-1}-\mathbf{e})\xi_2-\mathbf{u}^{-1}\mathbf{z}^{-1}\xi_3,
\end{align*}
where
\begin{align*}
\mathbf{g}&=\mathbf{z}^n[m+1]_{\mathbf{x}}-[m]_{\mathbf{x}}, \\
\mathbf{h}&=\mathbf{z}^{-n-1}\mathbf{x}^{-m}[n]_{\mathbf{z}}+[-n-1]_{\mathbf{z}}.
\end{align*}
Hence the space $\mathfrak{D}_\rho$ of derivations $H_{m,n}\to V_\rho$ can be identified with that of triples $(\xi_1,\xi_2,\xi_3)$ satisfying
$$\mathbf{u}\mathbf{z}^{-n-1}\mathbf{x}^{-m}\mathbf{g}\xi_1+(\mathbf{e}-\mathbf{z}^{-1})\xi_2+(\mathbf{z}^{-1}+\mathbf{u}\mathbf{h})\xi_3=0.$$

Similarly as in Section 2, $\chi_\rho\in\mathcal{J}_3(H_{m,n})$ if and only if $\dim\mathfrak{D}_\rho\ge 5$, which is equivalent to
\begin{align}
{\rm rank}(\mathbf{u}\mathbf{z}^{-n-1}\mathbf{x}^{-m}\mathbf{g},\mathbf{e}-\mathbf{z}^{-1},\mathbf{z}^{-1}+\mathbf{u}\mathbf{h})\le 1. \label{ineq:H-rank}
\end{align}
Necessarily, $\det(\mathbf{z}-\mathbf{e})=0$. Thus up to conjugacy we may assume that either $\mathbf{z}=\mathbf{d}(\lambda,1)$ with $\lambda \ne 1$, or $\mathbf{z}=\mathbf{p}$.

Suppose
$$\mathbf{x}=\left(\begin{array}{cc} a & b \\ c & d \end{array}\right), \qquad
\mathbf{u}=\left(\begin{array}{cc} a' & b' \\ c' & d' \end{array}\right),$$
with $ad-bc=1$ and $a'd'-b'c'=t\ne 0$.

As is necessary for (\ref{ineq:H-rank}), ${\rm rank}(\mathbf{e}-\mathbf{z}^{-1},\mathbf{z}^{-1}+\mathbf{u}\mathbf{h})\le 1$, which is the same as
\begin{align}
(c',d')\mathbf{h}=(0,-1).  \label{eq:H-rank=1}
\end{align}

\begin{rmk} \label{rmk:H-condition}
\rm For a reason similar to that in Remark \ref{rmk:G-condition}, (\ref{eq:H-rank=1}) will be also sufficient as long as $\mathbf{x}-\mathbf{e}$ is invertible.
\end{rmk}

Assume that (\ref{eq:H-relation}), (\ref{ineq:H-rank}) hold, so does (\ref{eq:H-rank=1}). As a consequence of (\ref{eq:H-relation}),
\begin{align}
{\rm tr}(\mathbf{v})={\rm tr}(\mathbf{z}). \label{eq:trace-v=z}
\end{align}

Let $\mu+\mu^{-1}=r=a+d$. For $k\in\mathbb{Z}$, put
\begin{align}
\gamma_k=\begin{cases}
(\mu^k-\mu^{-k})/(\mu-\mu^{-1}), &\mu\notin\{\pm 1\}, \\
k\mu^{k-1}, &\mu\in\{\pm 1\}.
\end{cases}  \label{eq:gamma-k}
\end{align}
We have
\begin{align}
\mathbf{x}^j=\gamma_j\mathbf{x}-\gamma_{j-1}\mathbf{e}, \qquad j\in\mathbb{Z}.   \label{eq:x-jth}
\end{align}
The following identities are useful:
\begin{align}
\gamma_{k+1}-r\gamma_k+\gamma_{k-1}=0, \nonumber  \\
\gamma_{k+1}^2+\gamma_k^2-r\gamma_{k+1}\gamma_k=1, \label{eq:identity-2} \\
\gamma_{k}^2-\gamma_{k+1}\gamma_{k-1}=1, \nonumber  \\
\gamma_{k+1}\gamma_k-\gamma_{k+2}\gamma_{k-1}=r. \nonumber
\end{align}
An alternative of (\ref{eq:identity-2}) is
\begin{align}
(2-r)\gamma_{k+1}\gamma_k=(1+\gamma_k-\gamma_{k+1})(1+\gamma_{k+1}-\gamma_k). \label{eq:identity-2'}
\end{align}

\begin{lem}
$\mathbf{z}\ne\mathbf{p}$.
\end{lem}
\begin{proof}
Assume $\mathbf{z}=\mathbf{p}$.
Similarly as in Section \ref{sec:G-p}, we have $c\ne 0$, and up to conjugacy can assume $a=0$, so that $bc=-1$ and $d=r$.

Using (\ref{eq:x-jth}) we can compute
\begin{align*}
\mathbf{v}&=\left(\begin{array}{cc}
(n+1)n\gamma_{m+1}\gamma_mc^2+(n\gamma_{m+1}^2-n-1)c & \star \\ c-n\gamma_{m+1}\gamma_mc^2 & r-n\gamma_{m+2}\gamma_mc \end{array}\right), \\
\mathbf{g}(\mathbf{x}-\mathbf{e})&=\left(\begin{array}{cc}  \gamma_{m-1}-\gamma_m+n\gamma_{m+1}c & (\gamma_{m+1}-\gamma_m)b+n(\gamma_{m+2}-1) \\
(\gamma_{m+1}-\gamma_m)c & \gamma_{m+2}-\gamma_{m+1}   \end{array}\right),
\end{align*}
where the $\star$ stands for something irrelevant.
The reason for considering $\mathbf{g}(\mathbf{x}-\mathbf{e})$ is that $\mathbf{g}(\mathbf{x}-\mathbf{e})$ has a simpler expression than $\mathbf{g}$.

From (\ref{eq:trace-v=z}) we obtain
\begin{align}
(n+1)n\gamma_{m+1}\gamma_mc^2-c=2-r. \label{eq:H-z-np-1}
\end{align}
It follows from $\det(\mathbf{g}(\mathbf{x}-\mathbf{e}))=0$, which is required by (\ref{ineq:H-rank}), that
\begin{align}
1+\gamma_m-\gamma_{m+1}=\frac{2-r}{nc}. \label{eq:H-z-np-2}
\end{align}
\begin{itemize}
  \item If $r=2$, then $\gamma_k=k$, and
        $$\mathbf{g}=\left(\begin{array}{cc} 1-m+nm(m+1)c/2 & mb+n(m+1)(m+2)/2 \\ mc & m+1 \end{array}\right),$$
        so $\det(\mathbf{g})=1-nm(m+1)c/2\ne 0$ by (\ref{eq:H-z-np-1}). This violates (\ref{ineq:H-rank}).
  \item If $r\ne 2$, then
        \begin{align*}
        (2-r+c)(2-r)&\stackrel{(\ref{eq:H-z-np-1})}=(n+1)nc^2\cdot(2-r)\gamma_{m+1}\gamma_m \\
        &\stackrel{(\ref{eq:identity-2'})}=(n+1)nc^2(1+\gamma_{m+1}-\gamma_m)(1+\gamma_{m}-\gamma_{m+1}) \\
        &\stackrel{(\ref{eq:H-z-np-2})}=(n+1)c(1+\gamma_{m+1}-\gamma_m)(2-r).
        \end{align*}
        Hence $$1+\gamma_{m+1}-\gamma_m=\frac{2-r+c}{(n+1)c},$$
        which together with (\ref{eq:H-z-np-2}) implies $2-r=nc$ and $\gamma_{m+1}=\gamma_m$. By (\ref{eq:identity-2}), $nc\gamma_m^2=1$. Then
        $$-n(c')^2(a'd'-b'c')^{-1}=(\mathbf{u}^{-1}\mathbf{z}^{-1}\mathbf{u})_{21}=\mathbf{v}_{21}=0,$$
        so $c'=0\ne d'$. By (\ref{eq:H-rank=1}), $\mathbf{h}_{21}=0$. On the other hand, as we can compute, $\mathbf{h}_{21}=-n\gamma_mc\ne 0$. This is a contradiction.
\end{itemize}

Therefore, $\mathbf{z}\ne\mathbf{p}$.
\end{proof}

From now on, assume $\mathbf{z}=\mathbf{d}(\lambda,1)$ with $\lambda\ne 1$. As in Section 2, let $\vartheta_k=\lambda^k-1$.

Computing directly,
\begin{align*}
\mathbf{v}_{11}&=\lambda^{-1}a+\lambda^{-n-1}\vartheta_n\gamma_{m+1}\gamma_mbc, \\
\mathbf{v}_{12}&=\big(\lambda^{-1}\gamma_{m+1}(\gamma_{m+1}-\gamma_ma)-\lambda^{-n-1}\gamma_m(\gamma_{m+1}d-\gamma_m)\big)b, \\
\mathbf{v}_{21}&=\big(\gamma_{m+1}(\gamma_{m+1}-\gamma_md)-\lambda^{n}\gamma_m(\gamma_{m+1}a-\gamma_m)\big)c, \\
\mathbf{v}_{22}&=d-\vartheta_n\gamma_{m+1}\gamma_mbc.
\end{align*}
The condition (\ref{eq:trace-v=z}) reads
\begin{align}
\lambda^{-1}(a-1)+(d-1)=\lambda^{-n-1}\vartheta_{n+1}\vartheta_n\gamma_{m+1}\gamma_mbc.  \label{eq:H-trace}
\end{align}

We can write $\mathbf{u}\mathbf{v}=\mathbf{z}^{-1}\mathbf{u}$ (which is equivalent to (\ref{eq:H-relation})) as
\begin{align}
(a',b')(\mathbf{v}-\lambda^{-1}\mathbf{e})=0, \label{eq:determine-u-1} \\
(c',d')(\mathbf{v}-\mathbf{e})=0. \label{eq:determine-u-2}
\end{align}
As a consequence,
\begin{align}
0&=d'(a',b')(\mathbf{v}-\lambda^{-1}\mathbf{e})=(t,0)(\mathbf{v}-\lambda^{-1}\mathbf{e})+b'(c',d')(\mathbf{v}-\lambda^{-1}\mathbf{e}) \nonumber \\
&=t(\mathbf{v}_{11}-\lambda^{-1},\mathbf{v}_{12})+b'(1-\lambda^{-1})(c',d').  \label{eq:determine-u-3}
\end{align}

Now
\begin{align*}
\mathbf{h}=\left(\begin{array}{cc}  \lambda^{-n-1}([n]_\lambda(\gamma_{m+1}-\gamma_ma)-[n+1]_\lambda) & -n\lambda^{-n-1}\gamma_mb \\
-[n]_\lambda\gamma_mc & n(\gamma_{m+1}-\gamma_md)-n-1 \end{array}\right).
\end{align*}
So (\ref{eq:H-rank=1}) reads
\begin{align}
\lambda^{-n-1}\big(\vartheta_n(\gamma_ma-\gamma_{m+1})+\vartheta_{n+1}\big)c'+\vartheta_n\gamma_mcd'=0, \label{eq:H-1} \\
n\lambda^{-n-1}\gamma_mbc'+(n(\gamma_md-\gamma_{m+1})+n+1)d'=1.  \label{eq:H-2}
\end{align}

\begin{lem} \label{lem:bc=0}
If $bc=0$, then $a=\lambda=d^{-1}$, and $c=0\ne b$.
\end{lem}
\begin{proof}
If $b=0$, then $\mathbf{v}_{12}=0$, and by (\ref{eq:determine-u-3}), $b'd'=0$. Since the irreducibility of $\rho$ requires $b'\ne 0$, we have $d'=0$. But this contradicts (\ref{eq:H-2}).

Hence $b\ne 0=c$, and the irreducibility of $\rho$ implies $c'\ne 0$. Then by (\ref{eq:determine-u-2}), $\mathbf{v}_{11}=1$, which implies $a=\lambda$.
\end{proof}

Consequently, the conjugacy indeterminacy of $\rho$ can be fixed by setting $b=1$, 
as we assume from now on. Then $c=ad-1$.

\subsection{$r=2$} \label{sec:r=2}

Write $a=1+s$ and $d=1-s$. Then $c=-s^2$. By Lemma \ref{lem:bc=0}, $s\ne 0$.

By (\ref{eq:H-trace}),
\begin{align}
(m+1)m\vartheta_{n+1}\vartheta_ns=\vartheta_{n+1}-\vartheta_n.  \label{eq:r=2-1}
\end{align}

Note that for each $j\in\mathbb{Z}$,
$$\mathbf{x}^j=\left(\begin{array}{cc} 1+js & j \\ -js^2 & 1-js \end{array}\right).$$
By direct computation,
\begin{align*}
\mathbf{g}=\left(\begin{array}{cc} (m+1)m\vartheta_ns/2+(m+1)\vartheta_n+1+ms & (m+1)m\vartheta_n/2+m \\ -ms^2 & 1-ms\end{array}\right).
\end{align*}
Then $\det(\mathbf{g})=0$, as required by (\ref{ineq:H-rank}), is equivalent to
\begin{align}
(m+1)m\vartheta_ns=2(m+1)\vartheta_n+2. \label{eq:r=2-2}
\end{align}

We can replace (\ref{eq:r=2-1}), (\ref{eq:r=2-2}) by
\begin{align}
(2m+1)\vartheta_{n+1}\vartheta_n+\vartheta_{2n+1}=0, \label{eq:r=2-3} \\
s=\frac{\vartheta_{n+1}-\vartheta_n}{(m+1)m\vartheta_{n+1}\vartheta_n}
=\frac{2(\vartheta_n-\vartheta_{n+1})}{m(\vartheta_{n+1}+\vartheta_n)},  \label{eq:r=2-4}
\end{align}
then use (\ref{eq:H-1}), (\ref{eq:H-2}) to find
\begin{align}
c'=-\frac{2\lambda^{n+1}}{(2n+1)(m+1)m\vartheta_{n+1}}, \qquad d'=\frac{1}{2n+1}  \label{eq:c'd'-when-r=2}
\end{align}
and use (\ref{eq:determine-u-3}) to find
\begin{align}
a'=\frac{(2n+1)t}{1-\lambda}\Big(\frac{1-\lambda}{m\vartheta_{-n}-1}-\frac{2\lambda^{n+1}}{(m+1)m}\Big), \qquad  b'=\frac{(2n+1)t}{1-\lambda}.   \label{eq:a'b'-when-r=2}
\end{align}

Conversely, suppose (\ref{eq:r=2-3})--(\ref{eq:a'b'-when-r=2}) hold. Then (\ref{eq:determine-u-1}), (\ref{eq:determine-u-2}) are satisfied; so is (\ref{eq:H-relation}). Actually (\ref{ineq:H-rank}) is also fulfilled. Now that
${\rm rank}(\mathbf{e}-\mathbf{z}^{-1},\mathbf{z}^{-1}+\mathbf{u}\mathbf{h})\le 1$ (which is equivalent to (\ref{eq:H-rank=1}))
has been ensured by (\ref{eq:c'd'-when-r=2}), to show (\ref{ineq:H-rank}), it remains to show
${\rm rank}(\mathbf{u}\mathbf{z}^{-n-1}\mathbf{x}^{-m}\mathbf{g},\mathbf{e}-\mathbf{z}^{-1})\le 1$.

To this end, let
$$\mathbf{k}=\mathbf{x}^m\mathbf{z}^{n+1}\mathbf{u}^{-1}(\mathbf{e}-\mathbf{z}^{-1})\mathbf{u}
\stackrel{(\ref{eq:H-relation})}=\mathbf{x}^m\mathbf{z}^{n+1}-\mathbf{z}^n\mathbf{x}^{m+1}.$$
Clearly, $\det(\mathbf{k})=0$. It is easy to verified that
$$\mathbf{k}_{\ast2}=\left(\begin{array}{cc} -(m+1)\vartheta_n-1 \\ s \end{array}\right)\parallel\mathbf{g}_{\ast1}.$$
Hence
${\rm rank}(\mathbf{g},\mathbf{k})\le 1$, which is equivalent to
${\rm rank}(\mathbf{u}\mathbf{z}^{-n-1}\mathbf{x}^{-m}\mathbf{g},\mathbf{e}-\mathbf{z}^{-1})\le 1$.

\subsection{$r\ne 2$}

By Remark \ref{rmk:H-condition}, the condition $\dim\mathfrak{D}_\rho\ge 5$ is equivalent to $(c',d')\mathbf{h}=(0,-1)$.

Let $\omega=1+\gamma_m-\gamma_{m+1}$.

\subsubsection{$\mathbf{v}_{22}\ne 1$}

By (\ref{eq:determine-u-2}), $(c',d')\parallel(1-\mathbf{v}_{22},\mathbf{v}_{12})$.
Direct computation reduces (\ref{eq:H-1}) to
\begin{align}
\vartheta_n\gamma_m(a-1)+(\vartheta_n\gamma_{m+1}-\vartheta_{n+1})(d-1)=\vartheta_{n+1}\vartheta_n\gamma_{m+1}\gamma_m(1-ad).  \label{eq:r-ne2}
\end{align}

\begin{lem} \label{lem:nonzero}
$\vartheta_{n+1}\vartheta_n\ne 0$, and $(\gamma_{m+1},\gamma_m)\ne(1,0), (0,-1)$.
\end{lem}

\begin{proof}
From (\ref{eq:H-trace}), (\ref{eq:r-ne2}) it is clear that $\vartheta_n\ne0$ and $(\gamma_{m+1},\gamma_m)\ne(1,0), (0,-1)$.

Assume $\vartheta_{n+1}=0$. Then $\vartheta_n=\lambda^{-1}-1$, and by (\ref{eq:H-trace}), (\ref{eq:r-ne2}),
$$a-1+\lambda(d-1)=0, \qquad \gamma_{m+1}=\lambda\gamma_m.$$
The first equation implies
$$a=1+\frac{r-2}{\lambda^{-1}-1}, \qquad d=1+\frac{r-2}{1-\lambda};$$
the second one together with (\ref{eq:identity-2}) implies
$\gamma_m^2=(\lambda^2-\lambda r+1)^{-1}$.
Then
\begin{align*}
\mathbf{v}_{22}-1&=d-1-\lambda\vartheta_n\gamma_m^2(ad-1) \\
&=\frac{r-2}{1-\lambda}-\frac{1-\lambda}{\lambda^2-\lambda r+1}\left(r-2-\frac{\lambda(r-2)^2}{(1-\lambda)^2}\right)=0,
\end{align*}
contradicting the assumption.
\end{proof}

This together with (\ref{eq:identity-2'}) implies $\omega\ne 0$.

A comparison of (\ref{eq:r-ne2}) with (\ref{eq:H-trace}) results in
$$\vartheta_n\gamma_m(a-1)+(\vartheta_n\gamma_{m+1}-\vartheta_{n+1})(d-1)=-\lambda^n(a-1)-\lambda^{n+1}(d-1),$$
leading us to
\begin{align}
a=1+(\vartheta_n^{-1}+\gamma_{m+1})(2-r)\omega^{-1}, \qquad
d=1+(\lambda^n\vartheta_n^{-1}+\gamma_{m})(r-2)\omega^{-1}. \label{eq:H-a-d-1}
\end{align}
Then (\ref{eq:H-trace}) can be converted into
\begin{align}
(\lambda^n\gamma_{m+1}-\gamma_m)((\gamma_{m+1}+\gamma_m)\vartheta_{n+1}\vartheta_n+\vartheta_{2n+1})=0. \label{eq:H-CJL0}
\end{align}
Here are the details: the left-hand-side of (\ref{eq:H-trace}) equals
$$(\lambda^{-1}\vartheta_{n+1}\vartheta_n^{-1}+\gamma_m-\lambda^{-1}\gamma_{m+1})(r-2)\omega^{-1},$$
and the right-hand-side equals
\begin{align*}
&\lambda^{-n-1}\vartheta_{n+1}\vartheta_n\gamma_{m+1}\gamma_m(r-2+(a-1)(d-1)) \\
=\ &\lambda^{-n-1}\vartheta_{n+1}\vartheta_n\gamma_{m+1}\gamma_m
\Big(2+\frac{(2-r)(\lambda^n\vartheta_n^{-1}+\gamma_m+\lambda^n\gamma_{m+1})}{\vartheta_n\omega}\Big)\frac{r-2}{\omega},
\end{align*}
where (\ref{eq:identity-2'}) is applied. Then (\ref{eq:H-trace}) becomes
\begin{align*}
&\lambda^n\vartheta_{n+1}\vartheta_n^{-1}+\lambda^{n+1}\gamma_m-\lambda^n\gamma_{m+1} \\
=\ &\vartheta_{n+1}\vartheta_n\big(2\gamma_{m+1}\gamma_m+(1+\gamma_{m+1}-\gamma_m)\vartheta_n^{-1}(\lambda^n\vartheta_n^{-1}
+\gamma_m+\lambda^n\gamma_{m+1})\big) \\
=\ &\vartheta_{n+1}\big(\lambda^n\vartheta_n^{-1}-\vartheta_n^{-1}\gamma_m+\lambda^{2n}\vartheta_n^{-1}\gamma_{m+1}
+(\lambda^n\gamma_{m+1}-\gamma_m)(\gamma_{m+1}+\gamma_m)\big),
\end{align*}
which is equivalent to (\ref{eq:H-CJL0}).

Since $\mathbf{v}_{22}\ne 1$, we have $\gamma_m\ne\lambda^n\gamma_{m+1}$. Thus
\begin{align*}
\gamma_{m+1}+\gamma_m+\vartheta_{2n+1}\vartheta_{n+1}^{-1}\vartheta_n^{-1}=0. 
\end{align*}

\begin{rmk}
\rm As an equation in $r$, it has $m$ solutions for generic $\lambda$.

We should still exclude the solutions with $\gamma_m=\lambda^n\gamma_{m+1}$ or $(\gamma_{m+1},\gamma_m)\in\{(1,0), (0,-1)\}$.
\end{rmk}

Using (\ref{eq:H-a-d-1}) we compute
\begin{align*}
\mathbf{v}_{12}
&=\lambda^{-1}(\lambda^{-n}\gamma_m-\gamma_{m+1}), \\
1-\mathbf{v}_{22}&=1-d+\vartheta_n\gamma_{m+1}\gamma_mc\stackrel{(\ref{eq:H-trace})}=\frac{\lambda^n(a-1)+d-1}{\vartheta_{n+1}}
=\frac{(2-r)(\lambda^n\gamma_{m+1}-\gamma_m)}{\vartheta_{n+1}\omega},
\end{align*}
and then, remembering $(c',d')\parallel(1-\mathbf{v}_{22},\mathbf{v}_{12})$, reduce (\ref{eq:H-2}) into
\begin{align}
c'=\frac{\lambda^{n+1}(r-2)}{(2n+1)\vartheta_{n+1}\omega}, \qquad d'=\frac{1}{2n+1}.   \label{eq:c'd'-1}
\end{align}

\subsubsection{$\mathbf{v}_{22}=1$}

In this case, $\vartheta_n\gamma_{m+1}\gamma_mc=d-1$, which combined with (\ref{eq:H-trace}) and Lemma \ref{lem:bc=0} implies $\vartheta_n\gamma_{m+1}\gamma_mc\ne 0$, and
\begin{align}
a=1+(2-r)\vartheta_n^{-1}, \qquad d=1+\lambda^n(r-2)\vartheta_n^{-1}. \label{eq:H-a-d-2}
\end{align}

Now that $1-\mathbf{v}_{11}=1-\lambda^{-1}\ne 0$, (\ref{eq:determine-u-2}) implies $(c',d')\parallel(\mathbf{v}_{21},1-\lambda^{-1})$.

It follows from $c=ad-1$ together with $r\ne 2$ that
$$\lambda^n+\lambda^{-n}=r+(\gamma_{m+1}\gamma_m)^{-1}\stackrel{(\ref{eq:identity-2})}=\gamma_{m+1}\gamma_m^{-1}+\gamma_{m}\gamma_{m+1}^{-1},$$
implying $\gamma_{m+1}=\lambda^{\pm n}\gamma_{m+1}$.
Note that
\begin{align*}
\mathbf{v}_{21}=(\gamma_{m+1}-\gamma_m)(\gamma_{m+1}-\lambda^n\gamma_m)c,  
\end{align*}
which, by (\ref{eq:H-1}), does not vanish. Hence $\gamma_{m+1}=\lambda^{-n}\gamma_{m}$.
As a result,
\begin{align*}
\lambda^{-n}\vartheta_n^2\gamma_m^2c=\frac{\mathbf{v}_{21}}{\lambda^n+1}=r-2=\lambda^{-n}\vartheta_n^2-\lambda^n\gamma_m^{-2},
\end{align*}
and (\ref{eq:H-a-d-2}) can be written in the same form as (\ref{eq:H-a-d-1}).

Assume $(c',d')=\kappa(\mathbf{v}_{21},1-\lambda^{-1})$.
Then we can deduce from (\ref{eq:H-1}) that
\begin{align*}
\gamma_m=\frac{-\lambda^n\vartheta_{2n+1}}{\vartheta_{2n}\vartheta_{n+1}}, \qquad
\gamma_{m+1}=\frac{-\vartheta_{2n+1}}{\vartheta_{2n}\vartheta_{n+1}},  
\end{align*}
and deduce from (\ref{eq:H-2}) that
\begin{align*}
\lambda\kappa^{-1}&=n\gamma_m(\lambda^{-n}+1)(r-2)+(\lambda-1)\big(n\gamma_m((r-2)\lambda^n\vartheta_n^{-1}+\lambda^{-n}\vartheta_n)+n+1\big) \\
&=(2n+1)(\lambda-1).
\end{align*}
Hence
\begin{align*}
c'=\frac{(\lambda^n+1)(r-2)}{(2n+1)\vartheta_{n+1}\omega}, \qquad d'=\frac{1}{2n+1},
\end{align*}
which is in the same form as (\ref{eq:c'd'-1}).

Furthermore, $\gamma_{m+1}+\gamma_m+\vartheta_{2n+1}\vartheta_{n+1}^{-1}\vartheta_n^{-1}=0$ is satisfied.

\subsection{The result}

Note that
$$\lim\limits_{r\to 2}\frac{2-r}{1+\gamma_m-\gamma_{m+1}}\stackrel{(\ref{eq:identity-2'})}=\lim\limits_{r\to 2}\frac{1+\gamma_{m+1}-\gamma_m}{\gamma_{m+1}\gamma_m}=\frac{2}{(m+1)m}.$$
Let
\begin{align*}
\delta_r=\begin{cases} (2-r)/(1+\gamma_m-\gamma_{m+1}), &r\ne 2, \\ 2/(m+1)m, &r=2. \end{cases}   
\end{align*}

The results of Section 3.1, 3.2.1 and 3.2.2 can be written uniformly.

Let $\sigma_{\lambda,r,\alpha}=\rho_{\mathbf{z},\mathbf{x},\mathbf{u}}$, with $\mathbf{z}=\mathbf{d}(\lambda,1)$,
\begin{align*}
\mathbf{x}&=\left(\begin{array}{cc} 1+\delta_r(\vartheta_n^{-1}+\gamma_{m+1}) & 1 \\
\delta_r(\delta_r(\gamma_m+\lambda^{2n+1}\vartheta_{n+1}^{-1}\vartheta_n^{-1})-2) & 1-\delta_r(\lambda^n\vartheta_n^{-1}+\gamma_m) \end{array}\right), \\
\mathbf{u}&=\left(\begin{array}{cc} \alpha(1-\lambda+\lambda\delta_r(\lambda^n\gamma_{m+1}-\gamma_m)\vartheta_{n+1}^{-1}) & \alpha(\lambda^{-n}\gamma_m-\gamma_{m+1}) \\ -\lambda^{n+1}\delta_r\vartheta_{n+1}^{-1}/(2n+1) & 1/(2n+1) \end{array}\right).
\end{align*}
Here we have put $\alpha=(2n+1)t/(1-\lambda)$ and determined $a',b'$ from (\ref{eq:determine-u-3}) and $a'd'-b'c'=t$.

Recall (\ref{eq:gamma-k}) that $\gamma_k$ is a function of $r$.

\begin{thm}
The cohomology jump locus $\mathcal{J}_3(H_{m,n})=\chi(\mathcal{F})$, where $\mathcal{F}\subset\mathcal{R}^{\rm irr}(H_{m,n})$ consists of $\sigma_{\lambda,r,\alpha}$'s such that
\begin{align*}
\alpha, \lambda\in\mathbb{C}^\ast, r\in\mathbb{C}, \quad \vartheta_{n+1}\vartheta_n\ne 0, \quad (\gamma_{m+1},\gamma_m)\ne (1,0),(0,-1), \\
\gamma_{m+1}+\gamma_m+\vartheta_{2n+1}\vartheta_{n+1}^{-1}\vartheta_n^{-1}=0.
\end{align*}
Under the identification $(H_{m,n})^\wedge\cong\mathbb{C}^\ast\times\mathbb{C}^\ast$ via $\tau\mapsto(\tau(z),\tau(u))$, the morphism
${\det}_\ast:\mathcal{J}_3(H_{m,n})\to(H_{m,n})^\wedge$ is given by
$${\det}_\ast([\sigma_{\lambda,r,\alpha}])=\Big(\lambda,\frac{1-\lambda}{2n+1}\alpha\Big)$$
and has image $\{\lambda\in\mathbb{C}^\ast\colon\vartheta_{n+1}\vartheta_n\ne 0\}\times\mathbb{C}^\ast$.
\end{thm}

The last assertion is guaranteed by
\begin{lem}
Suppose $2m+1+\vartheta_{2n+1}\vartheta_{n+1}^{-1}\vartheta_n^{-1}\ne 0$. Then the polynomial $h_\lambda(r):=\gamma_{m+1}+\gamma_m+\vartheta_{2n+1}\vartheta_{n+1}^{-1}\vartheta_n^{-1}$ has at least one root with
$(\gamma_{m+1},\gamma_m)\ne (1,0),(0,-1)$.
\end{lem}
\begin{proof}
When $m=1$, $(\gamma_{m+1},\gamma_m)=(r,1)\ne (1,0),(0,-1)$.

Suppose $m\ge 2$.
Obviously,
\begin{align*}
(\gamma_{m+1},\gamma_m)=(1,0)&\Leftrightarrow r\in\Big\{2\cos\Big(\frac{2j\pi}{m}\Big)\colon 1\le j\le \big[\frac{m}{2}\big]\Big\},  \\
(\gamma_{m+1},\gamma_m)=(0,-1)&\Leftrightarrow r\in\Big\{2\cos\Big(\frac{2j\pi}{m+1}\Big)\colon 1\le j\le \big[\frac{m+1}{2}\big]\Big\}.
\end{align*}
Also obvious is that for $r\ne 2$, 
$$h'_\lambda(r)=\frac{m\gamma_{m+1}-(m+1)\gamma_m}{r-2}.$$

If $h_\lambda$ has a root $r$ with $(\gamma_{m+1},\gamma_m)=(1,0)$, then $h_\lambda$ has $[m/2]$ such roots, each of which is single, and since $\vartheta_{2n+1}\vartheta_{n+1}^{-1}\vartheta_n^{-1}=-1$, it is impossible for $h_\lambda$ to have a root with $(\gamma_{m+1},\gamma_m)=(0,-1)$. Hence $h_\lambda$ has at least one root such that $(\gamma_{m+1},\gamma_m)\ne (1,0),(0,-1)$.

The case when $h_\lambda$ has a root $r$ with $(\gamma_{m+1},\gamma_m)=(0,-1)$ is similar.
\end{proof}

\medskip

\begin{proof}[Proof of Theorem \ref{thm:solution-H}]
Suppose there exists an isomorphism $\phi:H_{m',n'}\stackrel{\cong}\to H_{m,n}$.
Then for a generic point $\mathfrak{a}\in\det_\ast(\mathcal{J}_3(H_{m,n}))$,
$$m=\#{\det}_\ast^{-1}(\mathfrak{a})=\#{\det}_\ast^{-1}(\phi^\ast(\mathfrak{a}))=m',$$
and from ${\det}_\ast(\mathcal{J}_3(H_{m,n}))\cong {\det}_\ast(\mathcal{J}_3(H_{m',n'}))$ we see $n=n'$.
\end{proof}

\end{document}